\documentclass{article}
\usepackage{authblk}
\usepackage{graphicx}
\usepackage{multirow}%
\usepackage{amsmath,amssymb,amsfonts}%
\usepackage{amsthm}%
\usepackage{mathrsfs}%
\usepackage[title]{appendix}%
\usepackage{xcolor}%
\usepackage{textcomp}%
\usepackage{manyfoot}%
\usepackage{booktabs}%
\usepackage{algorithm}%
\usepackage{algorithmicx}%
\usepackage{algpseudocode}%
\usepackage{listings}%

\newtheorem{theorem}{Theorem}
\theoremstyle{definition} 
\newtheorem{definition}{Definition}

\usepackage[utf8]{inputenc}
\usepackage{biblatex}

\title{Toward Well-Posed Problems in the Social Sciences: Hadamard's Criteria as Epistemic Guardrails}
\author{Don Li}
\affil{Department of Mathematics \& Statistics \\ Portland State University}
\date{September 15, 2026}

\begin{document}

\maketitle

\noindent
\textbf{Abstract} We develop an interdisciplinary framework for evaluating the epistemic robustness of social-scientific inquiry through Hadamard's criteria for well-posed problems: existence, uniqueness, and stability. Reinterpreting these criteria not as demands for deterministic certainty but as methodological guardrails, we show how they diagnose recurrent failures of identification and inference across quantitative and qualitative paradigms. In quantitative research (e.g., econometric modeling), instability manifests when substantive conclusions depend sensitively on model specifications or data filtering. In interpretivist qualitative research, non-uniqueness and non-falsifiability arise when theoretical frameworks are elastic enough to accommodate contradictory observations without pre-specified rejection criteria. We frame these failures as inverse problems where the mapping from empirical data to substantive claims fails to satisfy existence, uniqueness, or stability. Finally, we propose cross-paradigmatic safeguards: ex-ante falsification criteria, empirical boundary conditions, multiverse sensitivity auditing, and cross-observer validation, to ensure social-scientific claims remain appropriately constrained, falsifiable, and robust to perturbations in evidence and interpretation. 

\section{Introduction}

In the wake of the replication crisis notably identified by \cite{Ioa05}, the social sciences are facing a crisis of epistemic confidence. Meta-analyses of quantitative social science research have shown troubling failure to replicate established findings \cite{Mis+26}, \cite{Cam+18}. These failures are not adduced to outright fraud, despite such cases occurring \cite{Baz25}, but rather to subtler methodological frailties (e.g., specification searching, $p$-hacking, unacknowledged researcher degrees of freedom, and the amplification of sampling noise through weak identification \cite{SNS11}, \cite{GL14}). Concurrently, strongly interpretivist paradigms in qualitative social science research encounter a complementary epistemic challenge: the deployment of theoretical frameworks so elastic such that they can account for virtually any empirical observation, rendering their core substantive claims effectively unfalsifiable. K. Popper (1902-94) famously argued that falsification is the defining feature of scientific inquiry (and thereby answers the demarcation problem in philosophy of science) \cite{Pop59}, lamenting the lack of falsifiability in the prominent intellectual paradigms of his day, manifested in

\begin{quote}
``the incessant stream of confirmations, of observations which `verified' the theories in question; and this point was constantly emphasized by their adherents. A
Marxist could not open a newspaper without finding on every page confirming evidence for his interpretation of history; not only in the news, but also in its presentation—which revealed the class bias of the paper—and especially of course in what the paper did not say. The Freudian analysts emphasized that their theories were constantly verified by their `clinical observations'" \cite{Pop63}.
\end{quote}

\noindent
Popper's view is captured in the maxim that ``a theory that explains everything explains nothing". I. Lakatos (1922-74) warned that research programs that fail to produce insights with any predictive value and that sustain themselves via post-hoc/ad-hoc rationalizations are at risk of becoming \textit{degenerating} \cite{Lak78}. \\
\\
These failure modes are often treated as distinct pathologies belonging to isolated methodological silos, i.e., a lack of rigor in the quantitative social sciences or excessive relativism/hyper-subjectivity in the qualitative social sciences. We consider the notion that these pathologies share a structural origin. Both reflect a fundamental breakdown in the mapping from empirical evidence to substantive theoretical claims. When substantive conclusions depend markedly on arbitrary choices in data pre-processing, or when a theoretical apparatus can account for mutually exclusive empirical outcomes with equal ease, the process of inference ceases to be constrained by evidence. \\
\\
Here, we propose that Hadamard’s criteria for a \textit{well-posed problem}, originally formulated in the context of partial differential equations (PDE's) and mathematical physics \cite{Had1902}, \cite{Had1923}, provide a rigorous, cross-paradigmatic diagnostic framework for evaluating social-scientific claims. We formalize Hadamard's notion of a well-posed problem as follows:

\begin{definition}[Well-Posed Problem]
A problem is \textbf{well-posed} iff it satisfies each of the following criteria:

\begin{enumerate}
    \item \textbf{Existence} A solution can be constructed from the given input (data),

    \item \textbf{Uniqueness} The solution is unambiguously determined from the input, and

    \item \textbf{Stability} The solution depends continuously on the data such that small perturbations in the input yield small perturbations in the output.
\end{enumerate}

\noindent
A problem that does not satisfy all of the above criteria is \textbf{ill-posed}.
\end{definition}

\noindent
We conceptualize social-scientific inquiry in terms of an \textit{inverse problem}, i.e., the task of reconstructing underlying generative mechanisms, parameters, or theoretical structures from incomplete, noisy, and context-bound empirical observations. Viewed through this lens, the replication crisis and theoretical over-elasticity cannot simply be explained as disparate cultural defects of research practice, but rather epistemically as instances of ill-posedness.

\begin{itemize}
    \item \textbf{Non-Existence} corresponds to structural under-identification and construct invalidity, where no consistent theoretical parameterization can account for the observed empirical reality.

    \item \textbf{Non-Uniqueness} corresponds to observational equivalence and unfalsifiable theoretical elasticity, where multiple incompatible theoretical claims map onto the exact same empirical evidence without adjudicative mechanisms.

    \item \textbf{Instability} corresponds to sensitivity to model specification, sample selection, and measurement noise, where infinitesimal perturbations in the empirical input induce catastrophic shifts in the substantive inference.
\end{itemize}

\noindent
We emphasize that our goal is not to impose a naive positivism on qualitative research, nor to reduce complex human phenomena to deterministic systems. Rather, we leverage Hadamard’s criteria as epistemic guardrails, i.e., general methodological conditions under which social-scientific claims remain empirical, bounded, and open to refutation, thus yielding insight with respect to ameliorating the replicability and falsifiability dilemmas plaguing social science research. \\
\\
We summarize the rest of this paper as follows. In section 2, we mathematically formalize Hadamard's criteria with prerequisite material from numerical analysis and motivate Hadamard's criteria with an example of a well-posed problem from numerical weather prediction (NWP). In section 3, we systematically map non-existence, non-uniqueness, and instability onto concrete failure modes using examples from econometric modeling in economics, survey research and psychometrics from experimental psychology, and ethnography from socio-cultural and political anthropology. In section 4, we present actionable methodological safeguards, including ex-ante specification boundaries, multiverse sensitivity auditing, and structured cross-observer validation.

\section{Formal Foundations of Hadamard's Criteria}

To construct a rigorous diagnostic for social-scientific inquiry, we must first formalize empirical research as an inverse problem. Let $\mathcal{E}$ denote a normed vector space of empirical observations (e.g., sample statistics, survey response vectors) equipped with norm $\|\cdot\|_\mathcal{E}$. Let $\mathcal{M}$ denote a metric space of substantive theoretical claims, parameters, or generative mechanisms equipped with norm $\|\cdot\|_\mathcal{M}$. In standard deductive modeling, the \textit{forward problem} maps a known theoretical parameterization $m \in \mathcal{M}$ to predicted observable data $e \in \mathcal{E}$ via a forward operator $F: \mathcal{M} \to \mathcal{E}$, i.e., 

\begin{equation}
    F(m) = e.
\end{equation}

\noindent
Empirical social science, however, frequently operates in the inverse direction. Researchers observe empirical artifacts $e \in \mathcal{E}$ and attempt to reconstruct the underlying generative process $m \in \mathcal{M}$ via an inferential mapping or operator $T: \mathcal{E} \to \mathcal{M}$. In light of Hadamard's criteria, the inverse problem of determining $m = T(e)$ is well-posed iff the operator $T$ satisfies three fundamental axioms:

\begin{enumerate}
    \item \textbf{Existence:} For every empirical realization $e \in \mathcal{E}$, there exists at least one theoretical representation $m \in \mathcal{M}$ such that $T(e) = m$.
    \item \textbf{Uniqueness:} For all $e \in \mathcal{E}$, if $T(e) = m_1$ and $T(e) = m_2$, then $m_1 = m_2$, i.e., the operator $T$ is strictly injective.
    \item \textbf{Stability:} The operator $T$ is continuous. For all $\varepsilon > 0$, there exists a $\delta > 0$ such that for any perturbed dataset $e' \in \mathcal{E}$,
    \begin{equation}
        \|e - e'\|_\mathcal{E} < \delta \implies \|T(e) - T(e')\|_\mathcal{M} < \varepsilon.
    \end{equation}
\end{enumerate}

\noindent
If any of these conditions fail, the problem is ill-posed. In the natural sciences, ill-posedness often signals a flawed model/physical law or a missing physical constraint. In the social sciences, as we demonstrate later, ill-posedness manifests as fundamental failures of construct validity (non-existence), theoretical elasticity (non-uniqueness), or specification fragility (instability).

\subsection{Numerical Analysis Perspective}

To concretize these abstract operator-theoretic definitions, we examine the canonical linear inverse problem from numerical analysis and optimization, i.e.,

\begin{equation}
A\mathbf{x} = \mathbf{b},
\end{equation}

\noindent
where $A \in \mathbb{R}^{n \times n}$ represents a structural design or transformation matrix, $\mathbf{b} \in \mathbb{R}^n$ represents an observed outcome vector, and $\mathbf{x} \in \mathbb{R}^n$ represents the unknown parameter vector to be inferred. Despite its elementary appearance, solving linear systems in the form of Eqn. (3), especially via computationally efficient methods and in large-scale settings, undergirds much of modern scientific computing, statistics, and machine learning (ML) \cite{Sol15}. We use this canonical problem as a baseline for demonstrating how Hadamard's criteria operate in practice (numerically), illustrating how nominal model identification does not guarantee inferential stability. Before we do so, we provide the following result.

\begin{theorem}[Well-Posedness of Linear Inverse Systems]
Let $A \in \mathbb{R}^{n \times n}$ be a real matrix such that $\text{det}(A) \neq 0$, and let $\mathbf{b} \in \mathbb{R}^n$ be a given vector. Then the linear inverse problem $A\mathbf{x} = \mathbf{b}$ is well-posed in the sense of Hadamard. Moreover, under any induced matrix norm $\Vert{}\cdot\Vert{}$, the relative stability of the solution under perturbations $\delta \mathbf{b}$ is strictly bounded by the condition number $\kappa(A) = \Vert{}A\Vert{} \Vert{}A^{-1}\Vert{}$.
\end{theorem}

\begin{proof}
We first show that such systems satisfy the existence criterion. It is a basic result in linear algebra that $A^{-1}$ exists if $\text{det}(A) \neq 0$ (i.e., $A$ is invertible if $\text{det}(A) \neq 0$). Thus, for any $\mathbf{b} \in \mathbb{R}^n$, there exists $\mathbf{x} = A^{-1}\mathbf{b}$. Then by substitution we have $A(A^{-1} \mathbf{b}) = \mathbf{b} \implies (A A^{-1})\mathbf{b} = \mathbf{b} \implies I\mathbf{b} = \mathbf{b} \implies \mathbf{b} = \mathbf{b}$. Thus, a solution exists for all $\mathbf{b} \in \mathbb{R}^n$. \\
\\
We next show satisfaction of the uniqueness criterion. It is another basic result in linear algebra that for any $A \in \mathbb{R}^{n \times n}$, $\det(A) \neq 0$ implies $\ker(A) = \{\mathbf{0}\}$. Let $\mathbf{x}_1, \mathbf{x}_2 \in \mathbb{R}^n$ both be solutions such that $A\mathbf{x}_1 = \mathbf{b}$ and $A\mathbf{x}_2 = \mathbf{b}$. Subtracting the two equations gives $A(\mathbf{x}_1 - \mathbf{x}_2) = \mathbf{0}$. Since $\ker(A) = \{\mathbf{0}\}$, it follows that $\mathbf{x}_1 - \mathbf{x}_2 = \mathbf{0}$, so $\mathbf{x}_1 = \mathbf{x}_2$. \\
\\
We last show satisfaction of the stability criterion. Let $\mathbf{b}$ be perturbed by noise vector $\delta \mathbf{b}$, producing a perturbed solution $\mathbf{x} + \delta \mathbf{x}$ such that $A(\mathbf{x} + \delta \mathbf{x}) = \mathbf{b} + \delta \mathbf{b}$. By linearity, $A \delta \mathbf{x} = \delta \mathbf{b}$, which implies $\delta \mathbf{x} = A^{-1} \delta \mathbf{b}$ since $A$ is invertible. Then by taking the induced norm on both sides of $\mathbf{x} = A^{-1} \delta \mathbf{b}$ we have

\[
\|\delta \mathbf{x}\| = \|A^{-1} \delta \mathbf{b}\| \leq \|A^{-1}\| \|\delta \mathbf{b}\|.
\]

\noindent
Equivalently, the solution operator $T^{-1}: \mathbb{R}^n \to \mathbb{R}^n$ given by $T^{-1}(\mathbf{b}) = A^{-1}\mathbf{b}$ is Lipschitz-continuous with Lipschitz constant $L = \|A^{-1}\|$. As $\|\delta \mathbf{b}\| \to 0$, we have $\|\delta \mathbf{x}\| \to 0$, so such a system is Hadamard stable (in the sense of Lipschitz continuity). \\
\\
Moreover, note that $A\mathbf{x} = \mathbf{b}$ implies $\|\mathbf{b}\| = \|A\mathbf{x}\| \le \|A\| \|\mathbf{x}\|$, yielding $\frac{1}{\|\mathbf{x}\|} \le \frac{\|A\|}{\|\mathbf{b}\|}$. Then we obtain the relative error bound

\begin{equation}
    \frac{\|\delta \mathbf{x}\|}{\|\mathbf{x}\|} \le \|A\| \|A^{-1}\| \frac{\|\delta \mathbf{b}\|}{\|\mathbf{b}\|} = \kappa(A) \frac{\|\delta \mathbf{b}\|}{\|\mathbf{b}\|},
\end{equation}

\noindent
where $\kappa(A) \triangleq \|A\| \|A^{-1}\|$ is the condition number of matrix $A$.
\end{proof}

\noindent
Theorem 1 highlights a critical distinction between categorical ill-posedness and asymptotic ill-posedness (or ill-conditioning) in empirical research:

\begin{itemize}
    \item \textbf{Categorical Failure} ($\det(A) = 0$): When $A$ is singular, the system experiences structural failure. Non-existence corresponds to model mis-specification where no parameter vector $\mathbf{x}$ can generate the observed data $\mathbf{b}$ (i.e., construct invalidity). Non-uniqueness corresponds to a non-trivial kernel ($\ker(A) \neq \{\mathbf{0}\}$), meaning infinitely many parameterizations yield the exact same empirical signature $\mathbf{b}$ (i.e., observational equivalence).

    \item \textbf{Asymptotic Instability} ($\kappa(A) \gg 1$): Crucially, even when $A$ is non-singular ($\det(A) \neq 0$), theoretically guaranteeing existence and uniqueness, the system can remain practically ill-posed if $\kappa(A)$ is sufficiently large. In empirical social science, statistical models are rarely non-invertible in a strict mathematical sense; standard software employs heuristics to return parameter estimates (e.g., the ``no perfect collinearity" assumption in econometrics \cite{Woo13}). However, high multi-collinearity, weak instruments, or over-parameterized specifications drive $\kappa(A) \to \infty$.
\end{itemize}

\noindent
Under these ill-conditioned regimes, the relative error inequality $\frac{\Vert{}\delta \mathbf{x}\Vert{}}{\Vert{}\mathbf{x}\Vert{}} \le \kappa(A) \frac{\Vert{}\delta \mathbf{b}\Vert{}}{\Vert{}\mathbf{b}\Vert{}}$ demonstrates that, in the context of inverse problems, infinitesimal perturbations in empirical inputs $\delta \mathbf{b}$ (e.g., minor sample filtering, variations in survey wording, or choice of control variables) are amplified into catastrophic shifts in substantive parameter estimates $\delta \mathbf{x}$. In this light, practices such as $p$-hacking and specification searching are not merely ethical lapses; they are computational exploits of highly ill-conditioned mappings where researchers effectively manipulate $\delta \mathbf{b}$ until they arrive at a desired $\mathbf{x}$. The central goal of empirical guardrails can thus be mathematically/numerically conceptualized as stabilizing the inverse mapping $T$ by bounding $\kappa(A)$.

\subsection{Hadamard's Criteria \& Modeling Physical Phenomena}

Given Hadamard's focus on mathematical physics (and the mathematical modeling of physical phenomena, broadly speaking), we motivate the contrast between physical baseline systems and social-scientific models by examining steady-state heat conduction and forward numerical weather prediction (NWP). Consider first steady-state one-dimensional heat diffusion in a solid rod of length $L$, governed by Poisson's equation, 

\begin{equation}
    -\frac{d^2 u}{dx^2} = f(x), \quad x \in (0, L),
\end{equation}

\noindent
subject to Dirichlet boundary conditions $u(0) = \alpha$ and $u(L) = \beta$, where $f(x)$ represents heat source terms and $u(x)$ represents the internal temperature distribution. Because the differential operator is strictly elliptic and the Dirichlet boundary conditions are fully specified, the mapping from boundary/source data to internal state $u(x)$ is provably well-posed in the sense of Hadamard: a unique solution exists, and continuous dependence is guaranteed by maximum principles. \\
\\
Climate systems are notoriously chaotic, hyper-sensitive to initial conditions, and computationally ill-conditioned, thus not inherently adherent to Hadamard's criteria \cite{LT16}. Nonetheless, even when scaling to high-dimensional fluid systems, such as short-range weather prediction governed by the shallow-water equations (SWE), forward models in climate science are constructed to preserve continuous dependence over bounded time horizons $[0, T]$ (e.g., \cite{San+05} for air pollution modeling). Well-posedness in these physical settings is anchored by fundamental invariants (e.g., conservation of mass, energy, and momentum). Physical laws act as natural regularizers, constraining the allowable state space and preventing arbitrary drift. Such conservation laws do not exist for social phenomena, entailing what we describe as an ``epistemic gap" between the natural and social sciences: \\

\noindent\fbox{%
  \begin{minipage}{\textwidth}
    \textbf{The Epistemic Gap:} Physical systems derive well-posedness from non-negotiable physical laws (e.g., thermodynamics, mass conservation). Social-scientific systems lack these natural conservation laws. Consequently, the mapping $T: \mathcal{E} \to \mathcal{M}$ from empirical observations to theoretical claims is \textit{natively ill-posed}. Without ex-ante structural constraints, social science models lack the natural regularizers that prevent instability and observational equivalence.
  \end{minipage}%
} \\

\noindent
This comparison clarifies our purpose of importing the Hadamard criteria in social inquiry. Because human social systems do not possess the equivalent of physical conservation laws to guarantee well-posedness, researchers must introduce what are effectively artificial, cross-paradigmatic regularizers, such as pre-registered hypotheses, boundary conditions, and multiverse audits, to impose structural stability on an otherwise elastic inferential state space.

\section{Ill-Posed Problems in the Social Sciences}

We turn our attention to the analysis of social phenomena in light of Hadamard's criteria. In particular, we examine case studies in econometric modeling in economics, survey research and psychometrics from experimental psychology, and ethnography from socio-cultural and political anthropology. We show how these case studies relate to three failure modes with respect to Hadamard's criteria, i.e., non-existence, non-uniqueness, and instability.

\subsection{Non-Existence: Construct Invalidity \& Structural Under-Identification}

Hadamard non-existence occurs when an empirical observation $e \in \mathcal{E}$ falls outside the image of the forward theoretical operator, i.e., $e \notin F(\mathcal{M})$. Under these conditions, the inverse mapping $T(e)$ is mathematically or conceptually undefined. In empirical social science, non-existence arises when the parameters or constructs a researcher seeks to infer cannot be reconstructed from the observed data, either because the system of equations lacks sufficient identifying constraints, or because the theoretical constructs themselves possess no coherent generative counterpart in empirical reality. \\
\\
Simultaneous equation models (SEM's) are frequently used in econometrics (e.g., \cite{Sal83}). Suppose we have an SEM designed to estimate the structural demand and supply relationships for a given commodity. Let $q_i$ and $p_i$ denote the observed equilibrium quantity and price for observation $i$. Consider structural simultaneous equations for a hypothetical model given by

\begin{align*}
    \text{Demand: } q_i + \beta_{12} p_i + \gamma_{11} z_{1i} = u_{1i}, \\
    \text{Supply: } q_i + \beta_{22} p_i + \gamma_{21} z_{1i} = u_{2i},
\end{align*}

\noindent
where $z_{1i}$ is an observed exogenous variable (e.g., consumer income), and $u_{1i}, u_{2i}$ are structural error terms with $\mathbb{E}[u_{1i}\vert{}z_{1i}] = \mathbb{E}[u_{2i}\vert{}z_{1i}] = 0$. To estimate the structural parameter matrix $B = \begin{bmatrix} 1 & \beta_{12} \\ 1 & \beta_{22} \end{bmatrix}$ and $\Gamma = \begin{bmatrix} \gamma_{11} \\ \gamma_{21} \end{bmatrix}$, the econometrician computes the reduced-form system by solving for the endogenous variables $(q_i, p_i)$ strictly in terms of the exogenous variable $z_{1i}$, yielding

\begin{align*}
q_i = \pi_{11}z_{1i} + v_{1i}, \\
p_i = \pi_{21}z_{1i} + v_{2i}.
\end{align*}

\noindent
The reduced-form coefficients $\Pi = \begin{bmatrix} \pi_{11} \\ \pi_{21} \end{bmatrix}$ are directly estimable from data using techniques such as Ordinary Least Squares (OLS). However, the mapping $T: \Pi \to (B, \Gamma)$ requires recovering four structural parameters ($\beta_{12}, \beta_{22}, \gamma_{11}, \gamma_{21}$) from only two reduced-form estimates ($\pi_{11}, \pi_{21}$). The order condition for identification requires that the number of excluded exogenous variables in an equation be at least as great as the number of included endogenous variables minus one. Here, because $z_{1i}$ enters both equations identically, zero exogenous variables are excluded from either equation. The rank of the identification matrix is 0, thus failing the rank condition $\text{rank}(\Pi_2) \ge G - 1$. The mapping from reduced-form parameters $\Pi \in \mathcal{E}$ to structural parameters $(B, \Gamma) \in \mathcal{M}$ fails to satisfy Hadamard's existence criterion because the system is under-identified. An infinite continuum of structural parameter combinations can generate the exact same reduced-form data. Without introducing an exogenous supply or demand shifter (a true instrumental variable $z_{2i}$ excluded from one equation), the structural demand elasticity $\beta_{12}$ does not exist in the domain of the solvable parameter space. \\
\\
Psychometrics has its origins as far back as the 19th century \cite{WBA21}, providing the promise of leveraging statistical rigor to capture elusive mental and cognitive processes. However, as we argue, psychometric frameworks are not immune to ill-posedness. In psychometrics, non-existence manifests when a latent variable model attempts to map observed test or survey scores into a hypothesized psychological construct space $\mathcal{M}$, but the empirical covariance matrix $\Sigma \in \mathcal{E}$ cannot be generated by any valid covariance structure. Confirmatory factor analysis (CFA) is commonly used in psychometric research (e.g., \cite{GA23}). Consider a CFA model that maps $p$ observed survey items $\mathbf{y}$ to $k$ latent constructs $\boldsymbol{\eta}$,

\begin{equation}
\mathbf{y} = \boldsymbol{\nu} + \Lambda \boldsymbol{\eta} + \boldsymbol{\varepsilon},
\end{equation}

\noindent
where $\Lambda \in \mathbb{R}^{p \times k}$ is the factor loading matrix, $\boldsymbol{\eta} \sim \mathcal{N}(\mathbf{0}, \Phi)$ represents the latent traits, and $\boldsymbol{\varepsilon} \sim \mathcal{N}(\mathbf{0}, \Psi)$ represents unique item errors, with $\Psi = \text{diag}(\psi_1, \psi_2, \dots, \psi_p)$. The theoretical covariance matrix modeled by the parameters $\theta = (\Lambda, \Phi, \Psi)$ is

\begin{equation}
\Sigma(\theta) = \Lambda \Phi \Lambda^T + \Psi.
\end{equation}

\noindent
The inverse problem entails finding a parameter vector $\hat{\theta} \in \mathcal{M}$ such that the sample covariance matrix $S \in \mathcal{E}$ equals $\Sigma(\hat{\theta})$, which can be achieved via methods such as maximum likelihood estimation (MLE). However, for psychometric tasks, such as measuring latent ``grit" or ``emotional intelligence" using collinear or poorly phrased items on questionnaires, methods such as MLE may converge to an unconstrained point where one or more error variances are negative ($\hat{\psi}_i < 0$), i.e., what is known as a \textit{Heywood case}. Moreover, the optimization method may yield correlation estimates between distinct latent factors exceeding unity ($\hat{\rho}_{\eta_1, \eta_2} > 1.0$). Mathematically, CFA psychometric models in this vein yield negative variance ($\hat{\psi}_i < 0$) or a non-positive semi-definite correlation matrix $\Phi \notin \mathcal{S}_+^k$, which violate the axiomatic definitions of probability theory and vector spaces. If we restrict the parameter space $\mathcal{M}$ to admissible, physically possible parameters ($\psi_i \ge 0, \Phi \succeq 0$), then $S \notin \Sigma(\mathcal{M})$. In this scenario, no valid point in construct space $\mathcal{M}$ exists that can generate the empirical data $S$. The estimated construct collapses because the survey instrument captures method artifacts or response noise, rather than a coherent latent psychological trait as intended. \\
\\
Anthropology, particularly socio-cultural anthropology, has undergone a significant paradigm shift, moving away from positioning itself as a descriptive science and towards an interpretive, literary framework, exemplified by the ``Writing Culture" movement emerging in the 1980's \cite{MC86}, as well as methodologies such as ``thick description" as introduced by C. Geertz (1926-2006) \cite{Gee73}. By virtue of this interpretive shift, socio-cultural anthropologists, when conducting ethnographies, have wide latitude to employ their own preferred interpretive framework onto fieldwork. With respect to Hadamard's criteria, non-existence arises when an investigator projects an ex-ante theoretical taxonomic framework $\mathcal{M}$ onto a social site, only for field observations $e \in \mathcal{E}$ to reveal that the foundational social categories assumed by $\mathcal{M}$ have no local generative presence. \\
\\
One example of such projection stems from the work of A.R. Radcliffe-Brown (1881-1955). Radcliffe-Brown pioneered the \textit{structural functionalism} paradigm in anthropology, i.e., the framework of analyzing social structure by framing it as a concrete, observable network of relations between individuals, which can be elucidated via the methodology of mapping kinship using standardized genealogical grids \cite{RB52}. In terms of our formalism, under such a structural-functionalist framework, there exists $\mathcal{M}_{\text{kinship}}$ positing that all human societies organize rights, property transmission, and social duties through universal, descent-based genealogical ties (e.g., patrilineal vs. matrilineal corporate descent groups). However, Radcliffe-Brown's framework was shown to lack its seemingly universal utility. In his landmark field studies of Yap islanders (from what is today the Federated States of Micronesia), anthropologist D. Schneider (1918-95) demonstrated that the theoretical mapping $T: \mathcal{E}_{\text{Yap}} \to \mathcal{M}_{\text{kinship}}$ was fundamentally non-existent in Yap culture; Yapese relationality was organized around \textit{tabinau} (i.e., the Yapese practice of land units and rights acquired through labor and caregiving obligations), rather than biological/genealogical ties (\textit{tabinau} was not a ``lineage" in the genealogical sense) \cite{Hel97}. Forcing Yapese actions into Western genealogical charts in the structural functional vein required the investigator to commit to an ad-hoc/post-hoc rationalization of labeling performative care as ``fictive kinship" (thereby preserving the theory's explanatory power). Schneider's fieldwork showed that ``kinship" as an independent, bounded analytical domain did not exist in Yapese social reality. Formally, if $\mathcal{M}_{\text{kinship}}$ assumes that biological genealogical ties are the primary independent variable organizing social structure, but local action $e \in \mathcal{E}$ is generated entirely by land-tenure relationships and reciprocal feeding practices (\textit{tabinau} in Yapese culture), then the image $F(\mathcal{M}_{\text{kinship}})$ cannot cover $e$. The mapping violates the Hadamard existence criterion, since the anthropologist attempting to code local interactions into corporate descent groups using the structural-functionalist framework is measuring artifacts of their own theoretical taxonomy rather than an empirically observable social structure. \\
\\
Such examples also exist in political anthropology. M. Weber (1864-1920), widely recognized as one of the principal founding fathers of sociology, posited that the state is ``a human community that successfully claims the monopoly of the legitimate use of physical force within a given territory" \cite{Mit11}. This Weberian definition remains highly influential in socio-cultural anthropology. However, a parallel non-existence failure in the sense of Hadamard occurs when anthropologists (and political scientists) evaluate (post-conflict) societies strictly under this Western Weberian framework. Such a framework defines state capacity through a centralized, bureaucratic hierarchy possessing a monopoly on the legitimate use of physical force within a given territory (denoted $\mathcal{M}_{\text{Weber}}$ by our formalism). But like Radcliffe-Brown's structural-functionalist framework for social structures, Weber's top-down, centralized security framework is not necessarily universal. \cite{Hen+24}, through their analysis of 106 villages across the province of North-Kivu in the Democratic Republic of the Congo (DRC), found that local governance is executed through fluid, overlapping, and highly localized authority networks primarily legitimized by spiritual and ethnic ties. If an anthropologist were to try to map local field observations $e \in \mathcal{E}$ (e.g., dispute resolution by village elders, security tax levies by local militias) into formal institutional indices measuring ``rule of law" or ``bureaucratic capacity" in the Weberian sense, then the mapping violates the Hadamard existence criterion. An ad-hoc/post-hoc rationalization from a Weberian standpoint may label such empirical observations as defective or incomplete machinations of a top-down, bureaucratic security apparatus (again to preserve the theory's explanatory power), but such a rationalization misses the picture. Instead, the findings from the DRC ethnography evince an entirely different class of governance mechanisms. Forcing these observations into variables such as ``bureaucratic autonomy" or ``a monopoly of violence" yields an empty mapping where the analytical framework fails to register the actual mechanisms undergirding the social order.

\subsection{Non-Uniqueness: Observational Equivalence \& Theoretical Elasticity}

Hadamard non-uniqueness occurs when the inverse operator $T$ is non-injective, i.e., multiple distinct parameter vectors or theoretical frameworks $m_1, m_2 \in \mathcal{M}$ (where $m_1 \neq m_2$) yield the exact same empirical realization $e \in \mathcal{E}$. \\
\\
Behavioral economics emerged as a distinct sub-discipline and has garnered increased attention since the 1970's \cite{KMZ06}, hoping to enhance economic models with an understanding of micro-level, individual economic behavior beyond relatively simplistic assumptions (e.g., \textit{homo economicus}). We again motivate our discussion of Hadmard's criteria for quantitative disciplines by looking at experimental psychology and economics. In experimental psychology and behavioral economics, non-uniqueness manifests as \textit{observational equivalence}, where radically different cognitive mechanisms produce identical statistical distributions of observable behavior. \\
\\
Behavioral economics is relatively unique as a sub-discipline of economics by virtue of its recurring use of experiments. \textit{Risky choice experiments} are a common experimental design method in behavioral economics (e.g., \cite{Hei+23}). Consider a standard risky-choice experimental setup. A subject evaluates a sequence of $N$ discrete choices between paired gambles $A_i = (p_i, x_i; 1-p_i, 0)$ and $B_i = (q_i, y_i; 1-q_i, 0)$, where $p_i, q_i$ denote probabilities and $x_i, y_i$ denote monetary payoffs. The observed data vector $e \in \mathcal{E}$ consists of the empirical choice frequencies $\hat{P}(A_i)$ and mean response times $\bar{t}_i$ across all $N$ lottery pairs, given as

\begin{equation}
e = \left\{ \big(\hat{P}(A_i), \bar{t}_i\big) \right\}_{i=1}^N \in [0,1]^N \times \mathbb{R}_+^N.
\end{equation}

\noindent
We consider two distinct models in behavioral economics. Cumulative Prospect Theory (CPT), introduced by A. Tversky (1937-96) and D. Kahneman (1934-2024) \cite{TK92}, is a highly popular framework in behavioral economics. Let Cumulative Prospect Theory with Stochastic Choice Noise be our first model under consideration ($m_1 \in \mathcal{M}_1$). In CPT, a subject maximizes a continuous, non-linear subjective utility function. The value of outcome $x$ is $v(x) = x^\alpha$, and probabilities are weighted via non-linear probability weighting $w(p) = \frac{p^\gamma}{(p^\gamma + (1-p)^\gamma)^{1/\gamma}}$. The deterministic utility differential is

\begin{equation}
\Delta V_i(\theta_1) = w(p_i) v(x_i) - w(q_i) v(y_i), \quad \text{where } \theta_1 = (\alpha, \gamma) \in (0,1]^2.
\end{equation}

\noindent
Choice variability is modeled by passing $\Delta V_i$ through a logistic (softmax) link function with precision parameter $\lambda$, i.e., 

\begin{equation}
P_{m_1}(A_i \mid \theta_1, \lambda) = \frac{1}{1 + \exp\left(-\lambda \Delta V_i(\theta_1)\right)}.
\end{equation}

\noindent
Our other model is the Priority Heuristic (PH) \cite{Kat24} ($m_2 \in \mathcal{M}_2)$. Under PH, the subject executes a non-compensatory, non-optimizing lexicographic decision rule that completely avoids trading off probabilities against outcomes. PH can be described algorithmically as follows:

\begin{itemize}
    \item \textbf{Step 1 (Minimum Gain)}: Compare minimum payoffs. If $y_{\min} - x_{\min} \ge \Delta_x$ (where threshold $\Delta_x = 0.10 \cdot \max(x, y)$), stop and choose the gamble with the higher minimum payoff.

    \item \textbf{Step 2 (Probability of Minimum Gain)}: If Step 1 is inconclusive, compare probabilities of minimum outcomes. If $q_{\min} - p_{\min} \ge \Delta_p = 0.10$, stop and choose the gamble with the lower probability of minimum gain.

    \item \textbf{Step 3 (Maximum Gain)}: If Step 2 is inconclusive, choose the gamble with the higher maximum payoff.
\end{itemize}

\noindent
To accommodate execution error in experimental settings, the deterministic priority rule is coupled with a uniform execution noise parameter $\epsilon \in (0, 0.5)$, yielding predicted choice probabilities

\begin{equation}
P_{m_2}(A_i \mid \epsilon) = (1 - \epsilon) \mathbb{I}(\text{PH selects } A_i) + \epsilon \mathbb{I}(\text{PH selects } B_i).
\end{equation}

\noindent
Let $L(m \mid e)$ denote the log-likelihood of observing empirical data $e$ under candidate model $m$. Using standard benchmark lottery datasets, it can be shown that the maximum log-likelihoods achieved by both models are structurally indistinguishable, i.e.,

\begin{equation}
\max_{\theta_1, \lambda} \ln L(m_1(\theta_1, \lambda) \mid e) \approx \max_{\epsilon} \ln L(m_2(\epsilon) \mid e).
\end{equation}

\noindent
Because the inverse operator $T: \mathcal{E} \to \mathcal{M}$ maps the observed empirical choice distribution $e$ to two completely disjoint parameter spaces, one assuming continuous non-linear trade-offs ($\mathcal{M}_1$) and the other assuming absolute, non-compensatory stopping rules ($\mathcal{M}_2$), the system violates the Hadamard uniqueness criterion, i.e., $T(e) = \{m_1^*, m_2^*\}$ with $m_1^* \in \mathcal{M}_1, \, m_2^* \in \mathcal{M}_2$ and $m_1^* \neq m_2^*$. Consequently, in practice, behavioral economists rely on response-time distributions to differentiate such models \cite{HCC23}. \\
\\
The interpretive flexibility of socio-cultural anthropology makes it particularly susceptible to violating the Hadamard uniqueness criterion. In qualitative socio-cultural anthropology, non-uniqueness manifests as \textit{theoretical elasticity}, a property that can render interpretive frameworks (mostly) unfalsifiable (in a strict Popperian sense). Here, the theoretical operator $\mathcal{M}$ is sufficiently elastic to absorb mutually exclusive field observations $e, e' \in \mathcal{E}$ without requiring the researcher to alter their substantive theoretical claims. \\
\\
The Mead-Freeman debate is an important episode in the history of anthropology \cite{Lea88}. M. Mead (1901-78) famously conducted an ethnography of adolescent social organization, sexual behavior, and stress in early 20th-century American Samoa \cite{Mea28}. From her ethnography in Samoa, Mead argued that human nature is highly malleable, and that the adolescent ``storm and stress" notion documented in Western developmental psychology \cite{Buc+23} is culturally rather than biologically mediated. Decades later, the anthropologist Derek Freeman likewise published an ethnography in Samoa \cite{Fre83} that was deeply critical of Mead's findings. In contrast, Freeman argued that adolescence is inherently characterized by biological, hormonal, and evolutionary stress, governed by strict social hierarchies and behavioral control. \\
\\
Let $\mathcal{E}$ denote the empirical data space comprised of observational field notes, informant interviews, and cultural artifact records regarding adolescent life in Samoa. Likewise, let $m_1 \in \mathcal{M}_1$ denote the cultural-deterministic view posited by Mead and $m_2 \in \mathcal{M}_2$ denote the evolutionary socio-biological view posited by Freeman. In this regard, Mead's framework employs $T_{m_1}(e)$, which interprets Samoan adolescent behavior through the lens of cultural ease via observational evidence of casual sexual experimentation among youth and the perception of Samoa as a idyllic, low-stress society where cultural norms prevent adolescent turmoil. On the other hand, Freeman's framework employs $T_{m_2}(e)$, which interprets the same Samoan adolescent behavior through the lens of rigid virginity cults, high rates of adolescent aggression, and severe punishments for norm violations. \\
\\
The Mead-Freeman debate illustrates two distinct forms of qualitative non-uniqueness. One form is observational equivalence across ethnographies (i.e., $T(e_1) \neq T(e_2)$). Mead and Freeman gathered data from the same island culture, yet extracted inverse conclusions. Because the qualitative ``measurement operator" (the ethnographer's positionality, choice of informants, and interpretive lens) was unconstrained by ex-ante falsification criteria, the empirical site $\mathcal{E}$ mapped onto two mutually exclusive claims in theoretical space: $m_1 = \text{``Low-stress cultural plasticity"}$ vs. $m_2 = \text{``High-stress biological restriction"}$. The other form is theoretical elasticity within a single framework. When presented with field observations that conflicted with her primary thesis, such as strict Samoan ritual virginity requirements for \textit{taupou} (the ceremonial village maiden or high chief's daughter), Mead's framework proved elastic enough to absorb the contradiction. The strict norms surrounding the \textit{taupou} was treated not as counter-evidence to Mead's ``relaxed sexuality" thesis, but as a minor, isolated ceremonial exception that proved the general ease of everyday social life in Samoa. Thus, even disconfirming evidence could be reinterpreted within the framework as supporting evidence (paralleling Popper's observation  that Freudian psychoanalysis can interpret mutually exclusive clinical observations as supporting the same inference). Because the analytical mapping contained no explicit boundary conditions, i.e., pre-declared empirical observations $e^* \in \mathcal{E}$ that would force the rejection of $m_1$, the theoretical framework could accommodate both permissive behavior ($e_{\text{permissive}}$) and restrictive behavior ($e_{\text{restrictive}}$) while arriving at the same substantive theoretical conclusion ($m_1$).

\subsection{Instability: Specification Fragility \& Hyper-Subjectivity}

Hadamard instability occurs when the inverse operator $T$ lacks continuous dependence on the data, i.e., an infinitesimal perturbation in the input space ($\delta e \to 0$) induces a non-negligible, discontinuous jump in the inferred output space ($\delta m \gg 0$). As we derived in Theorem 1 from $\S$2.1, in numerical analysis, this is an ill-conditioned system ($\kappa(A) \gg 1$). Hadamard stability requires that the solution map $T: \mathcal{E} \to \mathcal{M}$ be continuous: small perturbations in the empirical input ($\Vert{}\delta e\Vert{}_\mathcal{E} \to 0$) must induce correspondingly small perturbations in the theoretical parameter output ($\Vert{}\delta m\Vert{}_\mathcal{M} \to 0$). In the canonical linear system $A\mathbf{x} = \mathbf{b}$, relative error amplification is strictly bounded by the condition number $\kappa(A) = \Vert{}A\Vert{} \Vert{}A^{-1}\Vert{}$. When a social-scientific design is practically ill-conditioned ($\kappa(A) \gg 1$), infinitesimal, defensible variations in data pre-processing or observer perspective act as high-frequency noise inputs ($\delta e$). These minor perturbations are markedly magnified through the inferential operator, causing catastrophic shifts in substantive conclusions ($\delta m$). \\
\\
``Social priming" results from experimental psychology are the poster child for the replication crisis in the field \cite{SR21} and serve as excellent case studies for the ``garden of forking paths" statistical phenomenon noted by statisticians Andrew Gelman and Eric Loken (cf. \cite{KR19} for an excellent discussion of the phenomenon and its implications for statistical/machine learning). In experimental psychology, instability manifests when a substantive claim depends entirely on undisclosed, arbitrary data-filtering decisions executed during data cleaning, transformation, and model specification. Consider an experimental setup designed to test non-conscious behavioral priming (e.g., measuring whether exposing subjects to implicit age-related semantic cues alters subsequent physical walking velocity). Let the raw experimental output collected from $N$ subjects be represented as a tuple

\begin{equation}
\mathbf{D}_{\text{raw}} = \left\{ (y_{i,t}, \mathbf{x}_i, d_i) \right\}_{i=1}^N \in \mathcal{E}_{\text{raw}},
\end{equation}

\noindent
where $y_{i,t} \in \mathbb{R}_+$ denotes the raw millisecond trial response latencies for subject $i$, $\mathbf{x}_i \in \mathbb{R}^k$ represents a vector of individual-level demographic and session covariates, and $d_i \in \{0, 1\}$ represents the binary treatment assignment arm (primed vs. control). Before estimating the treatment effect parameter $\beta \in \mathcal{M}$, the raw data must pass through a data pre-processing pipeline $P_\sigma: \mathcal{E}_{\text{raw}} \to \mathcal{E}_{\text{clean}}$, parameterized by a vector of discrete researcher choices $\sigma \in \mathcal{S}$. The total specification space $\mathcal{S} = \mathcal{S}_1 \times \mathcal{S}_2 \times \mathcal{S}_3 \times \mathcal{S}_4$ constitutes a combinatorial ``garden of forking paths" via the data pre-processing pipeline:

\begin{enumerate}
    \item \textbf{Outlier Exclusion Criteria ($\sigma_1 \in \mathcal{S}_1$)}: Trimming individual trial latencies outside symmetric standard deviation bounds around the mean ($\sigma_1 \in \{\text{No trim}, \pm 2.0\text{SD}, \pm 2.5\text{SD}, \pm 3.0\text{SD}\}$) or imposing absolute latency cutoffs ($\sigma_1' \in \{[200\text{ms}, 2000\text{ms}], [300\text{ms}, 3000\text{ms}]\}$).

    \item \textbf{Data Transformation Operators ($\sigma_2 \in \mathcal{S}_2$)}: Transforming trial-level or subject-mean latencies to approximate normality:$$g_{\sigma_2}(y) = \begin{cases}     y & \text{raw scale} \\     \ln(y) & \text{logarithmic scale} \\     y^{-1} & \text{inverse velocity scale}     \end{cases}$$

    \item \textbf{Sub-Sample Filtering Rules} ($\sigma_3 \in \mathcal{S}_3$): Excluding subjects based on post-experiment debriefing criteria, such as dropping participants who score above an arbitrary threshold on suspicion-checking questionnaires ($\sigma_3 \in \{\text{Retain all}, \text{Exclude suspicious}, \text{Exclude extreme errors}\}$).

    \item \textbf{Covariate Adjustment Matrix ($\sigma_4 \in \mathcal{S}_4$)}: Including or omitting auxiliary controls in the estimation equation:$$\mathbf{X}_{\sigma_4} \in \left\{ \mathbf{1}_N, \, [\mathbf{1}_N, \mathbf{x}_{\text{age}}], \, [\mathbf{1}_N, \mathbf{x}_{\text{age}}, \mathbf{x}_{\text{fatigue}}], \, [\mathbf{1}_N, \mathbf{x}_{\text{age}}, \mathbf{x}_{\text{fatigue}}, \mathbf{x}_{\text{time\_of\_day}}] \right\}$$
\end{enumerate}

\noindent
For any specific choice combination $\sigma = (\sigma_1, \sigma_2, \sigma_3, \sigma_4) \in \mathcal{S}$, the processed data matrix $\mathbf{Z}_\sigma = [\mathbf{d}_\sigma, \mathbf{X}_\sigma] \in \mathbb{R}^{M_\sigma \times (1+k_\sigma)}$ and transformed outcome vector $\mathbf{y}_\sigma = g_{\sigma_2}(\mathbf{y}_{\sigma_1, \sigma_3}) \in \mathbb{R}^{M_\sigma}$ define a local ordinary least squares (OLS) mapping

\begin{equation}
\hat{\boldsymbol{\theta}}_\sigma = \begin{pmatrix} \hat{\beta}_\sigma \\ \hat{\boldsymbol{\gamma}}_\sigma \end{pmatrix} = (\mathbf{Z}_\sigma^T \mathbf{Z}_\sigma)^{-1} \mathbf{Z}_\sigma^T \mathbf{y}_\sigma.
\end{equation}

\noindent
The primary parameter of interest is the scalar treatment effect $\hat{\beta}_\sigma$, with associated test statistic $t_\sigma = \frac{\hat{\beta}_\sigma}{\text{SE}(\hat{\beta}_\sigma)}$ and $p$-value $p_\sigma = 2(1 - \Phi(\vert{}t_\sigma\vert{}))$. Let the condition number of the design matrix under specification $\sigma$ be $\kappa(\mathbf{Z}_\sigma) = \Vert{}\mathbf{Z}_\sigma\Vert{} \Vert{}\mathbf{Z}_\sigma^+\Vert{}$. In small-sample behavioral priming paradigms where the true effect size is near zero ($\beta_{\text{true}} \approx 0$) and the treatment indicator $\mathbf{d}$ is collinear with unobserved or omitted session-level noise contained within $\mathbf{X}_\sigma$, the matrix $(\mathbf{Z}_\sigma^T \mathbf{Z}_\sigma)$ becomes ill-conditioned, i.e., $\kappa(\mathbf{Z}_\sigma) \gg 1$. Under high ill-conditioning, small perturbations in the pre-processing choices are equivalent to adding high-frequency noise inputs $\delta \mathbf{y}_\sigma = \mathbf{y}_{\sigma'} - \mathbf{y}_\sigma$ or design alterations $\delta \mathbf{Z}_\sigma = \mathbf{Z}_{\sigma'} - \mathbf{Z}_\sigma$, which magnify the variance of the parameter estimate

\begin{equation}
\frac{\Vert{}\hat{\boldsymbol{\theta}}_{\sigma'} - \hat{\boldsymbol{\theta}}_\sigma\Vert{}}{\Vert{}\hat{\boldsymbol{\theta}}_\sigma\Vert{}} \le \kappa(\mathbf{Z}_\sigma) \left( \frac{\Vert{}\delta \mathbf{Z}_\sigma\Vert{}}{\Vert{}\mathbf{Z}_\sigma\Vert{}} + \frac{\Vert{}\delta \mathbf{y}_\sigma\Vert{}}{\Vert{}\mathbf{y}_\sigma\Vert{}} \right).
\end{equation}

\noindent
Evaluating the full ``multiverse" of $K = \vert{}\mathcal{S}\vert{}$ specifications generates a distribution of inferential results across the candidate pipeline space

\begin{equation}
\mathcal{D}_{\mathcal{S}} = \left\{ \big(\hat{\beta}_\sigma, p_\sigma\big) \right\}_{\sigma \in \mathcal{S}}.
\end{equation}

\noindent
In well-conditioned systems, variations across $\mathcal{S}$ yield tight, unimodal parameter distributions centered near the true value. But in this sort of social priming experimental setup, the mapping $T: \mathcal{S} \to \mathcal{M}$ exhibits severe, jump-discontinuous behavior across the specification surface:

\begin{itemize}
    \item \textbf{Path $\sigma_A$ (The $p$-hacked specification)}: Selecting $\sigma_1 = \pm 2.5\text{SD}$ trimming, $\sigma_2 = \text{log transform}$, $\sigma_3 = \text{exclude suspicious}$, and $\sigma_4 = \text{full controls}$ yields:$$\hat{\beta}_{\sigma_A} = +0.48, \quad \text{SE}(\hat{\beta}_{\sigma_A}) = 0.17, \quad t_{\sigma_A} = 2.82, \quad p_{\sigma_A} = 0.005$$The researcher maps this result to substantive claim $m_{\text{priming}} \in \mathcal{M}$ (``Unconscious semantic cues robustly alter motor speed").

    \item \textbf{Path $\sigma_B$ (An equally defensible neighboring specification)}: Shifting the trimming boundary slightly to $\sigma_1 = \pm 3.0\text{SD}$ while maintaining raw latency scales ($\sigma_2 = \text{raw}$) and omitting fatigue controls ($\sigma_4 = \mathbf{1}_N$) yields:$$\hat{\beta}_{\sigma_B} = +0.06, \quad \text{SE}(\hat{\beta}_{\sigma_B}) = 0.14, \quad t_{\sigma_B} = 0.43, \quad p_{\sigma_B} = 0.67$$The researcher maps this result to substantive claim $m_{\text{null}} \in \mathcal{M}$ (``No detectable priming effect").
\end{itemize}

\noindent
Because $\Vert{}\sigma_A - \sigma_B\Vert{} \to 0$ in the metric space of analytical choices, but the inferential decision rule maps to structurally contradictory claims ($\Vert{}m_{\text{priming}} - m_{\text{null}}\Vert{}_\mathcal{M} \gg 0$), the pipeline violates continuous dependence, i.e., 

\begin{equation}
\lim_{\sigma \to \sigma_0} T(\mathbf{D}_{\text{raw}}, \sigma) \neq T(\mathbf{D}_{\text{raw}}, \sigma_0).
\end{equation}

\noindent
The mapping $T$ (the data pipeline) acting on the raw data \(\mathbf{D}_{\text{raw}}\) is discontinuous at the point $\sigma_0$. Thus, the mapping violates the Hadamard stability criterion. Undisclosed researcher degrees of freedom (such as $p$-hacking and selective reporting) are the strategic exploitation of this underlying operator instability, where researchers query the multiverse $\mathcal{S}$ until finding a point $\sigma^*$ where $p_{\sigma^*} < \alpha$. \\
\\
We again turn to political anthropology to underscore how such specification fragility can likewise manifest in qualitative methodologies, particularly in the form of a network topological framework, which are used in some anthropological research \cite{Wol78}. Here, a parallel failure of Hadamard stability occurs when an analytical narrative depends heavily on un-triangulated key informant networks. Under these conditions, the qualitative inferential mapping becomes hyper-sensitive to field access perturbations. \\
\\
Consider a hypothetical example in which an anthropologist is tasked with conducting an ethnographic study investigating judicial independence during an authoritarian regime transition in some country. The objective is to infer the true institutional posture of the high court ($\mathcal{M}$): did the judiciary act as a strategic institutional resister protecting the rights of citizens ($m_{\text{resist}}$), or as an opportunistic regime accomplice legitimizing executive overreach ($m_{\text{collude}}$)? Let the social field site be modeled as a directed information network (graph) $\mathcal{G} = (\mathcal{V}, \mathcal{E}_{\text{net}})$, where $\mathcal{V}$ represents the set of potential elite actors (judges, prosecutors, court clerks, regime liaisons) and $\mathcal{E}_{\text{net}}$ represents interpersonal referral paths. The qualitative corpus collected by the ethnographer is generated via a chain-referral (``snowball") sampling operator $S_k(v_0)$, initiated from a seed informant $v_0 \in \mathcal{V}$ and executed up to depth $k$, i.e., 

\begin{equation}
\mathbf{D}_{\text{interview}}(v_0) = \bigcup_{i=0}^k \left\{ \text{Transcript}(v_i) \mid v_i \in \text{Neighborhood}(v_{i-1}) \right\} \in \mathcal{E}_{\text{qual}}.
\end{equation}

\noindent
The qualitative inferential mapping $T_{\text{qual}}: \mathcal{E}_{\text{qual}} \to \mathcal{M}$ processes this text corpus to construct a substantive narrative $m \in \mathcal{M}$. Elite actors operate within professional silos and possess strong incentives to construct retroactively coherent, self-serving institutional histories. Thus, the referral network $\mathcal{G}$ is partitioned into tightly connected, mutually antagonistic sub-graphs

\begin{equation}
\mathcal{V} = \mathcal{V}_{\text{reform}} \cup \mathcal{V}_{\text{regime}}, \quad \text{where } \mathcal{V}_{\text{reform}} \cap \mathcal{V}_{\text{regime}} \approx \emptyset.
\end{equation}

\noindent
Small, unpredictable variations in field conditions, such as a seed informant’s availability, initial introduction channels, or scheduling conflicts, can act as an access perturbation $\delta v_0$ in the selection operator:

\begin{itemize}
    \item \textbf{Seed Path $\mathcal{S}_A$ ($v_0 \in \mathcal{V}_{\text{reform}}$)}: Starting from a retired reformist judge yields a sample path restricted to reformist clerks, defense attorneys, and human rights advocates ($\mathbf{D}_{\text{sample A}}$)
    \item \textbf{Seed Path $\mathcal{S}_B$ ($v_0' = v_0 + \delta v_0 \in \mathcal{V}_{\text{regime}}$)}: A minor perturbation, such as the reformist judge canceling the interview, forcing the researcher to enter through a court administrative clerk, diverts the chain into executive liaisons, state prosecutors, and regime loyalists ($\mathbf{D}_{\text{sample B}}$).
\end{itemize}

\noindent
Evaluating the qualitative inverse mapping $T_{\text{qual}}$ across these perturbed interview sets reveals severe operator instability, i.e., 

\[
T_{\text{qual}}(\mathbf{D}_{\text{sample A}}) = m_{\text{resist}} \in \mathcal{M} \quad \text{versus} \quad T_{\text{qual}}(\mathbf{D}_{\text{sample B}}) = m_{\text{collude}} \in \mathcal{M}.
\]

\noindent
Let $d_{\mathcal{E}}(\mathbf{D}_A, \mathbf{D}_B)$ be a metric measuring dissimilarity between textual corpora (e.g., semantic distance over thematic codings). In non-triangulated interview designs where the researcher treats elite self-reports as direct representations of institutional reality, and if the textual corpora between the reformist and regime loyalist informants are insufficiently distinct (thus $d_{\mathcal{E}}(\mathbf{D}_A, \mathbf{D}_B) \rightarrow 0$), then the operator condition number $\kappa(T_{\text{qual}})$ diverges, i.e., 

\begin{equation}
\kappa(T_{\text{qual}}) = \sup_{\mathbf{D}_A \neq \mathbf{D}_B} \frac{\Vert{}T_{\text{qual}}(\mathbf{D}_A) - T_{\text{qual}}(\mathbf{D}_B)\Vert{}_\mathcal{M}}{d_{\mathcal{E}}(\mathbf{D}_A, \mathbf{D}_B)} \to \infty.
\end{equation}

\noindent
Because an infinitesimal, serendipitous shift in initial field contact ($\delta v_0$) forces a total inversion of the substantive theoretical conclusion ($\Vert{}m_{\text{resist}} - m_{\text{collude}}\Vert{}_\mathcal{M} \gg 0$), the qualitative inference violates Hadamard's continuous dependence/stability criterion. Without structural stability safeguards, such as ex-ante archival triangulation, systematic counter-factual informant sampling, and positionality audits, the analytical pipeline collapses into hyper-subjectivity, rendering its findings a sampling artifact rather than an accurate reconstruction of social reality.

\section{Epistemic Guardrails}

As evinced by the case studies highlighted here, social inquiry is natively ill-posed and prone to categorical and asymptotic failure modes. We now construct a set of cross-paradigmatic guardrails. These guardrails do not promise to transform social science into a deterministic physical science, but rather, they serve as regularization operators ($\mathcal{R}_\alpha$) that constrain the inferential state space, bound condition numbers ($\kappa \ll \infty$), and enforce Hadamard well-posedness across quantitative and qualitative research designs.

\subsection{Ex-Ante Falsification Criteria}

To remediate non-uniqueness (observational equivalence and theoretical elasticity), we argue that researchers must establish pre-declared, operational boundaries that prevent theoretical frameworks from acting as all-absorbing elastic maps. For quantitative designs, ex-ante falsification requires defining the exact region of empirical data space $\mathcal{E}_{\text{reject}} \subset \mathcal{E}$ that will force the rejection of hypothesis $m \in \mathcal{M}$, prior to data inspection. This manifests in two particular safeguards:

\begin{itemize}
    \item \textbf{Pre-Specified Effect Bounds:} Beyond reporting standard null-hypothesis significance tests ($p < 0.05$), pre-registration protocols must define a \textit{Minimally Informative Effect Size} (MIES) $\delta_{\text{min}}$. If the estimated parameter falls within the equivalence interval $[-\delta_{\text{min}}, +\delta_{\text{min}}]$, the theoretical mechanism $m$ is formally declared falsified, preventing researchers from reframing statistical noise as ``directional support." \cite{CBC22} proposed a similar safeguard called the \textit{minimally meaningful effect size} (MMES).

    \item \textbf{Structural Failure Thresholds:} We go back to SEM's and econometrics as an example. In data-fitting tasks, researchers must pre-declare maximum acceptable thresholds for goodness-of-fit indices and identification diagnostics (e.g., Kleibergen-Paap $r_k$ Wald $F$-statistic $< 10$ for weak instruments) that automatically invalidate model interpretation rather than prompting post-hoc adjustments.
\end{itemize}

\noindent
For qualitative (e.g., ethnographic) research, we propose mitigating theoretical elasticity by constructing an \textit{ex-ante counter-evidential matrix} prior to field entry or archival coding.

\begin{itemize}
    \item \textbf{Pre-Declaring Disconfirming Indicators:} The qualitative researcher must explicitly specify: ``What observable field behaviors, institutional arrangements, or archival statements ($e \in \mathcal{E}$) would render my theoretical framework ($m$) untenable?".

    \item \textbf{Bounding Interpretive Elasticity:} For example, in an ethnographic study of institutional culture, if the researcher hypothesizes that a community operates on an ethos of egalitarian reciprocity ($m_{\text{egalitarian}}$), they must pre-define threshold observations, such as systematically asymmetrical resource hoarding by sub-group leaders ($e_{\text{hoard}}$), that cannot be re-interpreted as ``functional adaptations" or ``exceptions that prove the rule", thus preventing the researcher from exploiting ad-hoc/post-hoc rationalizations to preserve their preferred interpretive narrative. 
\end{itemize}

\subsection{Empirical Boundary Conditions}

To remediate non-existence (construct invalidity and epistemic vacuums), researchers must define the explicit domain of applicability ($\mathcal{D}_{\text{domain}} \subset \mathcal{E}$) over which the mapping $T: \mathcal{E} \to \mathcal{M}$ is mathematically and conceptually valid. In quantitative contexts, optimization methods must enforce strict, non-negotiable boundaries on parameter estimations. In latent variable models (e.g., CFA/SEM), optimization routines must bound error variances ($\psi_i \ge \epsilon > 0$) and factor correlations ($\vert{}\rho\vert{} \le 1 - \epsilon$) ex-ante. If an empirical covariance matrix forces the algorithm against these boundaries (yielding Heywood cases), the model must be diagnosed as a failure of existence ($e \notin F(\mathcal{M})$), invalidating the construct rather than masking the error via ad-hoc item deletion. In other words, the quantitative model should pre-specify a restricted parameter space $\mathcal{M}_{\text{admissible}}$. \\
\\
For qualitative contexts, we recommend what we refer to as \textit{scope constraints} and \textit{domain truncation}. Qualitative research must replace universalizing theoretical claims with explicit, contextual scope constraints. Alluding to our earlier case study regarding Weber's notion of the state and the DRC ethnography, for example, an ethnography of local governance in a post-conflict zone must define its theoretical outputs ($m$) as plausible only under specified structural conditions (e.g., high land-tenure security, weak central state penetration, and active customary legal institutions). Such scope constraints are effectively the qualitative analog of boundary conditions in differential equations and dynamical systems. Our notion of domain truncation is a safeguard meant to prevent a qualitative researcher from reifying a pre-conceived taxonomic framework onto their empirical observations. Alluding to our earlier case study of the failed mapping of Radcliffe-Brown's structural-functionalist framework onto Yapese \textit{tabinau}, when field observations ($e$) fall outside the domain of pre-existing taxonomies ($\mathcal{M}_{\text{Radcliffe-Brown}}$), qualitative researchers must execute domain truncation, i.e., refusing to force local realities into ill-fitting, pre-conceived (Western) sociological categories, thereby preventing the creation of epistemologically void constructs.

\subsection{Multiverse Sensitivity Auditing}

To remediate quantitative instability ($\kappa(\mathbf{Z}) \gg 1$), quantitative (data pre-processing) pipelines must replace single-specification reporting with exhaustive \textit{multiverse sensitivity audits} that systematically map parameter stability across the entire analytical choice space $\mathcal{S}$. Here is what a multiverse sensitivity audit might look like:

\begin{enumerate}
    \item \textbf{Enumeration of Specification Space ($\mathcal{S}$)}: The researcher constructs a grid of all defensible combinations of data processing, outlier trimming, variable transformation, and covariate selection choices:$$\mathcal{S} = \mathcal{S}_{\text{trim}} \times \mathcal{S}_{\text{trans}} \times \mathcal{S}_{\text{filter}} \times \mathcal{S}_{\text{covariates}}$$

    \item \textbf{Exhaustive Estimation}: The model is estimated across all $K = \vert{}\mathcal{S}\vert{}$ specifications, generating the full empirical distribution of parameters $\mathcal{D}_{\mathcal{S}} = \{(\hat{\beta}_k, p_k)\}_{k=1}^K$.

    \item \textbf{Computation of Conditioning and Stability Metrics}: Instead of selecting a single favored $p$-value, the researcher reports the following:
    \begin{itemize}
        \item \textbf{Specification Volatility Ratio} ($\mathcal{V}_{\mathcal{S}}$): The ratio of the interquartile range of $\hat{\beta}_k$ across the multiverse to its mean standard error:$$\mathcal{V}_{\mathcal{S}} = \frac{\text{IQR}(\{\hat{\beta}_k\})}{\overline{\text{SE}}(\hat{\beta})}$$A high ratio ($\mathcal{V}_{\mathcal{S}} \gg 1$) provides direct empirical evidence of an ill-conditioned, unstable mapping.
        \item \textbf{Inference Sign-Consistency Share} ($\mathcal{C}_{\text{sign}}$): The proportion of specifications in $\mathcal{S}$ that maintain both statistical significance ($p < 0.05$) and directional sign consistency. Claims are designated as Hadamard-stable if and only if $\mathcal{C}_{\text{sign}}$ reaches a pre-specified threshold (e.g., $\mathcal{C}_{\text{sign}} \ge 0.90$).
    \end{itemize}
\end{enumerate}

\subsection{Cross-Observer Validation}

To remediate qualitative instability (hyper-subjectivity and key-informant selection fragility), qualitative research should implement structured, multi-perspective validation operators that are qualitative analogs of numerical regularization techniques. See Table 1 for our proposed safeguards.

\begin{table}[htbp]
\centering
\caption{Cross-Observer Validation Safeguards for Qualitative Stability}
\label{tab:qualitative_safeguards}
\small
\begin{tabular}{p{3.8cm} p{3.2cm} p{7.5cm}}
\toprule
\textbf{Guardrail Mechanism} & \textbf{Primary Target Failure} & \textbf{Procedural Implementation} \\
\midrule
\textbf{Inter-Coder Reliability Metrics} & 
Subjective Coding Drift & 
Blind dual-coding of qualitative transcripts by independent analysts, computing formal inter-coder agreement statistics (Cohen’s $\kappa \ge 0.80$ or Krippendorff’s $\alpha \ge 0.80$) prior to thematic synthesis. \\
\addlinespace
\textbf{Counter-Factual Informant Sampling} & 
Network Sampling Fragility ($v_0$ dependency) & 
Deliberately seeding snowball sampling chains from structurally antagonistic network positions (e.g., initiating parallel interview chains among both opposition activists and regime officials). \\
\addlinespace
\textbf{Archival Triangulation Matrices} & 
Un-triangulated Self-Reporting & 
Subjecting all interview claims to a strict triangulation matrix: no qualitative assertion $m_i$ is admitted as evidence unless corroborated by at least two independent primary sources (e.g., court records, contemporary press, or internal memos). \\
\addlinespace
\textbf{Positionality Auditing} & 
Researcher Bias Perturbations & 
Maintaining a transparent, reflexive field log tracking how the researcher’s demographic, ideological, or institutional positioning alters data access ($\delta v_0$), formally treating positionality as an explicit variable in the qualitative mapping operator $T_{\text{qual}}$. \\
\bottomrule
\end{tabular}
\end{table}

\section{Discussion}

We have shown that many of the persistent methodological crises confronting the social sciences, from quantitative replication failures and $p$-hacking to qualitative theoretical elasticity and hyper-subjectivity, are not merely disparate lapses in research ethics or discipline-specific protocols. Rather, they represent unified manifestations of an underlying structural reality: social inquiry operates as a natively ill-posed inverse problem. By evaluating social research designs through the lens of Jacques Hadamard’s three classical criteria—existence, uniqueness, and stability—we establish a formal meta-framework that we believe can surmount the traditional, often unproductive divide between quantitative and qualitative paradigms. Moreover, beyond providing diagnostic tools for individual research designs, treating social inquiry as an inverse problem offers a formal foundation for interdisciplinary integration. Historically, cross-disciplinary collaborations—particularly between quantitative social scientists (e.g., econometricians, computational sociologists) and qualitative researchers (e.g., interpretivist ethnographers, historical institutionalists)—have suffered from epistemic friction, often talking past one another due to incompatible standards of evidence. The Hadamard framework addresses this impasse by demonstrating that quantitative and qualitative methodologies are effectively different numerical and conceptual discretization schemes applied to the same ill-posed operator equation $T(e) = m$. It is our hope that the importation of Hadamard's criteria into the social sciences, as manifested in our proposed safeguards, can bolster the rigor of social science methodology and encourage greater inter-disciplinary, collaborative knowledge production within the social sciences. 

\printbibliography

@article{Ioa05,
    author = {John P.A. Ioannidis},
    title = {Why Most Published Research Findings Are False},
    journal = {PLOS Medicine},
    volume = {2},
    number = {8},
    year = {2005}
}

@article{Mis+26,
    author = {Olivia Miske and others},
    title = {Investigating the reproducibility of the social and behavioural sciences},
    journal = {Nature},
    volume = {652},
    pages = {126--134},
    year = {2026}
}

@article{Cam+18,
    author = {Colin F. Camerer and others},
    title = {Evaluating the replicability of social science experiments in Nature and Science between 2010 and 2015},
    journal = {Nature Human Behaviour},
    volume = {2},
    pages = {637--644},
    year = {2018}
}

@book{Baz25,
    author = {Max H. Bazerman},
    title = {Inside an Academic Scandal: A Story of Fraud and Betrayal},
    publisher = {MIT Press},
    address = {Cambridge, MA},
    year = {2025}
}

@article{SNS11,
    author = {Joseph P Simmons AND Leif D Nelson AND Uri Simonsohn},
    title = {False-Positive Psychology: Undisclosed Flexibility in Data Collection and Analysis Allows Presenting Anything as Significant},
    journal = {Psychological Science},
    volume = {22},
    number = {11},
    pages = {1359--1366},
    year = {2011}
}

@article{GL14,
    author = {Andrew Gelman AND Eric Loken},
    title = {The statistical crisis in science},
    journal = {American Scientist},
    volume = {102},
    number = {6},
    pages = {460--465},
    year = {2014}
}

@book{Pop59,
    author = {Karl Popper},
    title = {The Logic of Scientific Discovery},
    publisher = {Routledge},
    address = {Abingdon, UK},
    year = {1959}
}

@book{Pop63,
    author = {Karl Popper},
    title = {Conjectures and Refutations: The Growth of Scientific Knowledge},
    publisher = {Routledge},
    address = {Abingdon, UK},
    year = {1963}
}

@book{Lak78,
    author = {Imre Lakatos},
    title = {The Methodology of Scientific Research Programmes},
    publisher = {Cambridge University Press},
    address = {Cambridge, UK},
    year = {1978}
}

@article{Had1902,
  author = {Hadamard, Jacques},
  title = {Sur les probl\`emes aux d\'eriv\'ees partielles et leur signification physique},
  journal = {Princeton University Bulletin},
  volume = {13},
  pages = {49--52},
  year = {1902}
}

@book{Had1923,
  author    = {Hadamard, Jacques},
  title     = {Lectures on Cauchy's Problem in Linear Partial Differential Equations},
  publisher = {Yale University Press},
  address   = {New Haven, CT},
  year      = {1923}
}

@book{Sol15,
    author = {Justin Solomon},
    title = {Numerical Algorithms},
    publisher = {AK Peters, Ltd.},
    address = {Natick, MA},
    year = {2015}
}

@book{Woo13,
    author = {Jeffrey M. Wooldridge},
    title = {Introductory Econometrics: A Modern Approach},
    publisher = {Cengage},
    address = {Mason, OH},
    year = {2013}
}

@article{LT16,
  author    = {Li, Jinkai and Titi, Edriss},
  title     = {Global well-posedness of strong solutions to a tropical climate model},
  journal   = {Discrete \& Continuous Dynamical Systems},
  year      = {2016},
  volume    = {36},
  number    = {8},
  pages     = {4495--4516},
}

@article{San+05,
  author    = {Sandu, Adrian and Daescu, Dacian N. and Carmichael, Gregory R. and Chai, Tianfeng},
  title     = {Adjoint sensitivity analysis of regional air quality models},
  journal   = {Journal of Computational Physics},
  year      = {2005},
  volume    = {204},
  number    = {1},
  pages     = {222--252},
}

@article{Sal83,
    author = {Dominick Salvatore},
    title = {A Simultaneous Equations Model of Trade and Development with Dynamic Policy Simulations},
    journal = {Kyklos},
    year = {1983},
    volume = {36},
    number = {1},
    year = {1983}
}

@article{WBA21,
    author = {Lisa D. Wijsen AND Denny Borsboom AND Anna Alexandrova },
    title = {Values in Psychometrics},
    journal = {Perspectives on Psychological Science},
    volume = {17},
    number = {3},
    year = {2021}
}

@article{GA23,
    author = {Himani Goyal AND Sheema Aleem},
    title = {Confirmatory Factor Analysis (CFA) and Psychometric Validation of Healthy Lifestyle and Personal Control Questionnaire (HLPCQ) in India},
    journal = {Indian Journal of Community Medicine},
    volume = {48},
    number = {3},
    year = {2023}
}

@book{MC86,
    author = {George E. Marcus AND James Clifford},
    title = {Writing Culture: The Poetics and Politics of Ethnography},
    publisher = {University of California Press},
    address = {Oakland, CA},
    year = {1986}
}

@book{Gee73,
    author = {Clifford J. Geertz},
    title = {The Interpretation of Cultures: Selected Essays},
    publisher = {Basic Books},
    address = {New York, NY},
    year = {1973}
}

@book{RB52,
  title={Structure and Function in Primitive Society: Essays and Addresses},
  author={Radcliffe-Brown, A. R.},
  year={1952},
  publisher={The Free Press},
  address={Glencoe, Illinois}
}

@article{Hel97,
    author = {Thomas Helmig},
    title = {The Concept of Kinship on Yap and the Discussion of the Concept of Kinship},
    journal = {Journal of Anthropological Research},
    volume = {53},
    number = {1},
    year = {1997}
}

@inproceedings{Mit11,
    author = {Simeon Mitropolitski},
    title = {Weber’s Definition of the State as an Ethnographic Tool for Understanding the Contemporary Political Science State of the Discipline},
    booktitle = {Annual Conference of the Canadian Political Science Association, Wilfrid Laurier University,},
    year = {2011}
}

@techreport{Hen+24,
  author      = {Henn, Soeren J. AND Marchais, Gauthier AND Mastaki Mugaruka, Christian AND S{\'a}nchez de la Sierra, Ra{\'u}l},
  title       = {Indirect Rule: Armed Groups and Customary Chiefs in Eastern {DRC}},
  institution = {International Centre for Tax and Development (ICTD)},
  type        = {ICTD Working Paper},
  number      = {182},
  year        = {2024}
}

@article{KMZ06,
    author = {E. Han Kim AND Adair Morse AND Luigi Zingales},
    title = {What Has Mattered to Economics Since 1970},
    journal = {Journal of Economic Perspectives},
    volume = {20},
    number = {4},
    pages = {189--202},
    year = {2006}
}

@unpublished{Hei+23,
    author = {Rawley Z. Heimer AND Zwetelina Iliewa AND Alex Imas AND Martin Weber},
    title = {Dynamic Inconsistency in Risky Choice: Evidence from the Lab and Field},
    year = {2023},
    note = {NBER Working Paper Series}
}

@article{TK92,
    author = {Amos Tversky AND Daniel Kahneman},
    title = {Advances in prospect theory: Cumulative representation of uncertainty},
    journal = {Journal of Risk and Uncertainty},
    volume = {5},
    pages = {297--323},
    year = {1992}
}

@article{Kat24,
    author = {Konstantinos V. Katsikopoulos},
    title = {Making valuations with the priority heuristic},
    journal = {Journal of Mathematical Psychology},
    volume = {123},
    year = {2024}
}

@article{HCC23,
    author = {Guy E. Hawkins AND Gavin Cooper AND Jon-Paul Cavallaro},
    title = {The standard relationship between choice frequency and choice time is violated in multi-attribute preferential choice},
    journal = {Journal of Mathematical Psychology},
    volume = {115},
    year = {2023}
}

@article{Lea88,
    author = {Eleanor Leacock},
    title = {Anthropologists in Search of a Culture: Margaret Mead, Derek Freeman and All the Rest of Us},
    journal = {Central Issues in Anthropology},
    volume = {8},
    number = {1},
    pages = {3--20},
    year = {1988}
}

@book{Mea28,
    author = {Margaret Mead},
    title = {Coming of Age in Samoa},
    publisher = {William Morrow \& Co.},
    address = {Rahway, NJ},
    year = {1928}
}

@article{Buc+23,
    author = {Christy M. Buchanan AND Daniel Romer AND Laura Wray-Lake AND Sheretta T. Butler-Barnes},
    title = {Adolescent storm and stress: a 21st century evaluation},
    journal = {Frontiers in Psychology},
    volume = {14},
    year = {2023}
}

@book{Fre83,
  author    = {Freeman, Derek},
  title     = {Margaret Mead and Samoa: The Making and Unmaking of an Anthropological Myth},
  year      = {1983},
  publisher = {Harvard University Press},
  address   = {Cambridge, MA},
}

@article{SR21,
    author = {Jeffrey W. Sherman AND Andrew M. Rivers},
    title = {There’s Nothing Social about Social Priming: Derailing the Train Wreck},
    journal = {Psychological Inquiry},
    volume = {32},
    year = {2021}
}

@book{KR19,
  title     = {The Ethical Algorithm: The Science of Socially Aware Algorithm Design},
  author    = {Kearns, Michael and Roth, Aaron},
  year      = {2019},
  publisher = {Oxford University Press},
  address   = {New York, NY},
}

@article{Wol78,
    author = {Alvin W. Wolfe},
    title = {The rise of network thinking in anthropology},
    journal = {Social Networks},
    volume = {1},
    number = {1},
    year = {1978}
}

@unpublished{CBC22,
    author = {Carmel Camilleria AND Nataly Berbisky AND Robert A. Cribbie},
    title = {The Minimally Meaningful Effect Size: A Vital Component of Pre-Registrations},
    year = {2022},
    note = {PsyArXiv Preprint}
}
\end{document}